\documentclass[reqno, a4paper, 10pt]{amsart}

\usepackage[utf8x]{inputenc}

\usepackage[dvips]{epsfig}
\usepackage{amsgen, amstext,amsbsy,amsopn,amssymb, amsthm, amsfonts,amssymb,amscd,amsmath,euscript,enumerate,url,verbatim,calc}
\usepackage{textcomp,xypic}

\newtheorem{lemma}{Lemma}[section]
\newtheorem{theorem}[lemma]{Theorem}
\newtheorem{cor}[lemma]{Corollary}
\newtheorem{prop}[lemma]{Proposition}

\newtheorem{defn}[lemma]{Definition}
\newtheorem{notn}[lemma]{Notation}

\def\E{{\rm E}}
\def\S{{\rm S}}
\def\H{{\rm H}}
\def\ZZ{\mbox{{\rm Z\hskip-4pt \rm Z}}}  

\title[Structure theorem for the Elementary Symplectic Group]{Quillen-Suslin theory 
for a structure theorem for the Elementary Symplectic Group}

\author{Neeraj~Kumar} 
\address{Indian Institute of Technology Bombay, Powai 400 076, India}
\email{neeraj@math.iitb.ac.in}
\curraddr{Dipartimento di Matematica, Universit\'{a} di Genova, Via Dodecaneso 35, 16146 Genova, Italy}
\email{kumar@dima.unige.it}

\author{Ravi A.~Rao} 
\address{Tata Institute of Fundamental Research, Mumbai 400 005, India}
\email{ravi@math.tifr.res.in}

\thanks{ 
\endgraf
}

\begin{document}
\hfuzz25pt

\begin{abstract}
A new set of elementary 
symplectic elements is described. It is shown that these also generate the elementary symplectic 
group {\rm ESp}$_{2n}(R)$. These generators are more symmetrical than the 
usual ones, and are useful to study the action of the elementary symplectic 
group on unimodular rows. Also, an alternate proof of, {\rm ESp}$_{2n}(R)$ 
is a normal subgroup of {\rm Sp}$_{2n}(R)$, is shown using the Local 
Global Principle of D. Quillen for the new set of generators.
\end{abstract}

\maketitle

\section{\large Introduction}

Let $R$ be a commutative ring with $1$. 
The symplectic group Sp$_{2n}(R)$ is the isotropy group of $\psi_n $ under 
the action of SL$_{2n}(R)$ on $\psi_n$ by conjugation, i.e. Sp$_{2n}(R) =
\{\alpha \in $ SL$_{2n}(R)\mid \alpha^t \psi_n \alpha = \psi_n\}$.
(Here $\psi_n \in$~SL$_{2n}(\ZZ)$ denotes the usual alternating
matrix (of Pfaffian one) got by placing the standard $2\times 2$
alternating matrix  $\psi_1$ of Pfaffian one, diagonally $n$
times.

The Elementary Group $\text{E}_n(R)$ is the subgroup of 
$\text{GL}_n(R)$ generated by $\{ \E_{ij}(\lambda) : \lambda \in R \}$, 
for $i \neq j $, where $\E_{ij}(\lambda)= I_n + \lambda e_{ij}$
and $e_{ij}$ is the matrix with $1$ on the $(i, j)$-{th} position 
and $0$ elsewhere. $I_{n}$ denotes the identity matrix.

The Elementary Symplectic group ESp$_{2n}(R)$ is the subgroup of
Sp$_{2n}(R)$ generated by the following ``symplectic elementary"
matrices $\E_{2 1}(x), \E_{12}(x)$, for $x \in R$, $\S_{ij}(\lambda)$, $1
\leq i \neq j \neq \pi (i) \leq 2n$, $\lambda \in R$, where $\pi$ is the
permutation $(1~2)(3~4) \cdots (2n - 1~2n) \in$~ S$_{2n}$:
\begin{eqnarray*}
\S_{ij}(\lambda) &=& I_{2n} + \lambda e_{ij} -(-1)^{i + j}\lambda
e_{\pi(j)\pi(i)}.
\end{eqnarray*}

In this note we describe a more ``symmetric'' set of generators
denoted by $\E(A)$, $\E(B)$, $\E(C)$, $\E(D)$, which generate  ESp$_{2n}(R)$.

These generators are useful in analysing the action of the 
elementary group on a unimodular row. A recent result in \cite{pr}, 
which generalizes a famous lemma of L.N. Vaser{\v{s}}te{\u\i}n in (\cite{SV}, Lemma 5.6), 
states 
that the elementary orbit of a unimodular row coincides with its elementary 
symplectic orbit. L.N. Vaser{\v{s}}te{\u\i}n used this result in \cite{SV} to prove that 
the orbit space of unimodular rows of length three modulo elementary action 
is isomorphic to the elementary symplectic Witt group W$_{\E}(R)$ over a two 
dimensional ring $R$. Such an isomorphism was also shown for a non-singular 
affine algebra of dimension three over an algebraically closed field $k$, when 
characteristic $k \neq 2, 3$ by Rao-van der Kallen in \cite{vdkr}.

The theorem in this note enables one to study elementary 
symplectic action, in a more symmetric way. The objective of the 
second named author was to consider if the Vaser{\v{s}}te{\u\i}n symbol from the orbit space
Um$_3(A)/$E$_3(A) \to $W$_\E(A)$ is injective, when 
$A$ is a three dimensional affine algebra over an algebraically closed field 
(even if $A$ is singular). This leads to the study of the action of a 
$1$-stably elementary matrix on an alternating matrix; 
which is discussed in \cite{frs}, It is the hope of the second author that the 
structure theorem developed here will be useful to analyse this. 

\paragraph{{\bf Convention:}} In this article, we assume that $R$ is a commutative ring
with $1$ and $2$ is an invertible element in $R$.

\medskip

\section{\large The initial structure theorem}

To describe the initial structure theorem we isolate the following
four types of  basic elementary symplectic generators $\E(A)$,
$\E(B)$, $\E(C)$, $\E(D)$ respectively which we will use in the
sequel. We begin with a notation.

\begin{notn}
{\rm We denote by
\begin{eqnarray*}
\E\begin{pmatrix} \lambda a & \lambda b\\ \mu a & \mu
b\end{pmatrix}:&=& \begin{pmatrix} I_{2} & X \\ \psi_{n - 1} X^t
\psi_{1} & I_{2n - 2}\end{pmatrix},
\end{eqnarray*}
for some $X \in$~ M$_{2, 2n-2}(R)$ of the form $X = (X_3, X_5,
\cdots X_{2n-1})$, where $$X_{2i - 1} =
\begin{pmatrix}\lambda a_i &\lambda b_i\\\mu a_i & \mu
b_i\end{pmatrix},$$ for some $\lambda, \mu, a_i, b_i \in R $, for
$2 \leq i \leq n$. This is a symplectic matrix since $\det(X_i) =
0$, for all $i$.

 We can also write this as $ \E((X_3)_2,(X_5)_3 \cdots (X_{2n-1})_n)$
 where the lower indices indicate that the block $ X_{2i-1} $ lies
 in the $(1,i)$-th position of the block matrix $\E\begin{pmatrix}\lambda a
& \lambda b\\ \mu a & \mu b\end{pmatrix} \in $ Sp$_n(M_2(R))$.

  We will denote by $\E^k\begin{pmatrix}\lambda a
& \lambda b\\ \mu a & \mu b\end{pmatrix}$, a matrix of the above
type which has {\it \bf precisely} the $X_{2k - 1}$ block that is 
non-zero. (The $k$ will usually be clear from the context). Note also 
that this is the same as $\E(X_k)$.}
\end{notn}

For $\alpha \in {\rm SL}_r(R)$, $\beta \in {\rm SL}_s(R)$, $\alpha \perp \beta$
is the matrix $\begin{pmatrix}\alpha & 0\\0 & \beta\end{pmatrix} \in 
{\rm SL}_{r+s}(R)$.

Examples are stated below in the form of a lemma:

\begin{lemma}  \label{1}
For $n > 1$, we have 
\begin{eqnarray*}
\E_{2 1} (\lambda) {\rm S}_{1 2n-1}(x) {\rm S}_{1 2n}(y)\E_{2 1}(-\lambda)
& = & (\delta \perp I_{2n - 2})\E^{n}\begin{pmatrix} x & y\\ \lambda x &
\lambda y \end{pmatrix},
\end{eqnarray*}
for $\delta = \E_{21}(\lambda)\E_{12}(xy)\E_{21}(-\lambda)\in 
\E_2(R)$. \hfill $\square$
\end{lemma}
\begin{lemma}\label{2}
A symplectic matrix \E$\begin{pmatrix}\lambda x & \lambda y\\
\mu x & \mu y\end{pmatrix}$ is elementary symplectic if \\ $(\lambda, \mu)
\sim_{{\rm E}_{2}(R)} (1, 0)$.
\end{lemma}
Proof: Let $X_{2k - 1} \neq 0$ in the given matrix, and let $(\lambda,
\mu) \varepsilon^{{-1}^t} = (1, 0)$. Then
\begin{eqnarray*}
{\rm E}\begin{pmatrix}\lambda x & \lambda y\\ \mu x & \mu y\end{pmatrix} = (\varepsilon \perp
I_{2n - 2}) \E_{12}(-xy)\S_{1 2k - 1}(x) \S_{1 2k}(y) (\varepsilon \perp I_{2n -
2})^{-1}.
\end{eqnarray*}
\hfill $\square$

The four basic elementary symplectic matrices are
$$\E\begin{pmatrix} a & a\\ a & a\end{pmatrix},\; \E \begin{pmatrix} b &
-b \\ b & -b\end{pmatrix},\; \E \begin{pmatrix} c & c \\ -c & -c\end{pmatrix},\;
\E\begin{pmatrix}-d &
d \\ d & -d\end{pmatrix}.
$$

We will denote these by type $\E(A(a))$, $\E(B(b))$, $\E(C(c))$, $\E(D(d))$
respectively; or just by  $\E(A)$, $\E(B)$, $\E(C)$, $\E(D)$ when we are
not too concerned about the actual entries, but only interested in the
shape and form under discussion.\\\\
{\bf Definition of H.} Let $R$ be a commutative ring. We define $H (=H(R))$ to be the 
subgroup of {\rm ESp}$_{2n}(R)$ generated by elements of the type 
$\E(A)$, $\E(B)$, $\E(C)$, $\E(D)$. 

\iffalse
\begin{defn} Let $R$ be a commutative ring. We define $H (=H(R))$ to be the 
subgroup of {\rm ESp}$_{2n}(R)$ generated by elements of the type 
$\E(A)$, $\E(B)$, $\E(C)$, $\E(D)$. 
\end{defn}
\fi
We shall also refer the $2 \times 2$ matrices
$$\begin{pmatrix} a & a \\ a & a\end{pmatrix}, \begin{pmatrix} b & -b \\
b & -b \end{pmatrix},
\begin{pmatrix} c & c \\ -c & -c \end{pmatrix}, \begin{pmatrix} -d & d \\
d & -d \end{pmatrix},$$ to be of type $A(a)$, $B(b)$, $C(c)$,
$D(d)$ respectively; or just $A$, $B$, $C$, $D$ to indicate the
shape. For future reference we may use the relations:

\begin{lemma} \label{abcd} Let $x, y \in R$. Then
\begin{eqnarray*}
A(x)A(y)=&A(2xy), \;  &A(x)B(y) = B(2xy).\nonumber\\
A(x)C(y)=&0 ,          &A(x)D(y) = 0. \nonumber\\
B(x)A(y)=&0 ,          &B(x)B(y) =0.\nonumber\\
B(x)C(y)=&A(2xy) ,     &B(x)D(y) = B(-2xy).\nonumber\\
C(x)A(y)=&C(2xy) ,     &C(x)B(y) = D(-2xy).\nonumber\\
C(x)C(y)=&0 ,          &C(x)D(y) = 0.\nonumber\\
D(x)A(y)=&0 ,          &D(x)B(y) = 0.\nonumber\\
D(x)C(y)=&C(-2xy),     &D(x)D(y) = D(-2xy).\nonumber
\end{eqnarray*} \hfill$\square$
\end{lemma}
We will find it convenient to denote type $\E(A)$ as $\begin{pmatrix} A(a) \\
D(a) \end{pmatrix}$, type $\E(B)$ as $\begin{pmatrix}B(b) \\ B(b)
\end{pmatrix}$, type $\E(C)$ as
$\begin{pmatrix} C(c) \\ C(c) \end{pmatrix}$, type $\E(D)$ as
$\begin{pmatrix} D(d) \\ A(d) \end{pmatrix}$ as it allows us to
keep in focus the shape of the ``horizontal" and ``vertical"
components of a basic elementary symplectic matrix. Or else think
of them as the ``top" and ``bottom" component as in the case when
$n = 2$.

One can go further to use this notation: the reader will
understand if we write for $X$, $Y \in$~ M$_2(R)$,
$\begin{pmatrix}X \\ 0 \end{pmatrix}$, $\begin{pmatrix} 0 \\ X
\end{pmatrix}$, and also $\begin{pmatrix} X \\ Y \end{pmatrix}$,
with some indication for the placement of the blocks $X$, $Y$. In
particular, one could have $\begin{pmatrix} B \\ B \end{pmatrix}$
with the transpose of the top $B$ and the bottom $B$ not being the
same matrix and only being matrices having the same  form. Note
that in the extended notation some of the matrices are obviously
not elementary symplectic of course. All this is just to say that
the special elementary symplectic matrices have the {\em
``splitting property"}:

\begin{lemma} \label{3} {\rm (Splitting property)}
For $X \in$~ {\rm M}$_2(R)$, with $\det(X) = 0$,
\begin{eqnarray*}
\E(X) = \begin{pmatrix}X \\ \psi_1 X^t \psi_1\end{pmatrix}
= \begin{pmatrix}X \\ 0\end{pmatrix}\begin{pmatrix} 0
\\ \psi_1 X^t \psi_1 \end{pmatrix} = \begin{pmatrix} 0 \\
\psi_1 X^t\psi_1\end{pmatrix}\begin{pmatrix} X \\ 0 \end{pmatrix}.
\end{eqnarray*} \hfill $\square$
\end{lemma}

\iffalse
(We will call this the decompostion of $\E(X)$ into its ``top" (the horizontal
one) ${\rm T}(X)$ and ``bottom" (the vertical one) ${\rm B}(X)$ parts. )
\fi

We will henceforth use the notations, for $\delta \in {\rm SL}_2 (R)$,
$I_2 + X \in \E_2(R)$, $X = B$ or $C$; $\sigma \in$ SL$_{2n}(R)$ etc. 
\begin{eqnarray*}
^{\delta}\sigma &=& (\delta \perp I_{2n - 2})\sigma, \\
\sigma^{\delta}
&=&
\sigma ( \delta \perp I_{2n - 2}),\\
\sigma_{I_2 + X} &=& \sigma (I_{2k -2} \perp (I_2 + X) \perp I_{2n
- 2k}),
\end{eqnarray*}
for some $k$, which will be clear from the context.

\begin{lemma} \label{4}
For $\delta \in {\rm SL}_2 (R)$, $ y_i \in
R$, $3 \leq i \leq 2n$,
\begin{eqnarray*}
^{\delta} {\displaystyle ({\prod_{i=3}^{2n}}} \S_{1 i}(y_i))^{\delta^{-1}} & = &
^{\sigma}\prod_{i = 2}^n \E^{i}
\begin{pmatrix} \lambda y_{2i - 1} & \lambda y_{2i}\\
\mu y_{2i - 1} & \mu y_{2i}
\end{pmatrix},
\end{eqnarray*}
where $(\lambda, \mu) = e_1 \delta^{t}$, $\sigma = \delta
\E_{12}(y_3y_4 +y_5y_6+ \cdots + y_{2n-1}y_{2n})\delta^{-1}$.
 \hfill $\square$
\end{lemma}

\begin{lemma} \label{5} Let $R$ be a commutative ring. Then one
has the identities:
$$
\E\begin{pmatrix} \lambda x & \lambda y \\ \mu x & \mu y \end{pmatrix} =
^{Ch} \E \begin{pmatrix} \lambda a & \lambda a \\ \mu a & \mu a\end{pmatrix}
\E \begin{pmatrix} \lambda b & -\lambda b \\ \mu b & - \mu b \end{pmatrix},
\eqno{(1)}
$$
with $a  =  \frac{ x +  y}{2}$, $b = \frac{x - y}{2}$, and where $Ch$ is the
matrix
\begin{eqnarray*}
Ch &=& \begin{pmatrix} 1 - 2\lambda \mu ab &  2\lambda^2 ab\\ -2\mu^2 ab & 1 +
2\lambda \mu ab\end{pmatrix} \in {\rm SL}_2 (R) \cap~{\E}_3 (R).\\
& = & \varepsilon (I_2 + 2ab e_1^t e_2)\varepsilon^{-1} \in {\rm SL}_2(R),
\end{eqnarray*}
if $(\lambda, \mu) = e_1 \varepsilon^t$, for some $\varepsilon \in
{\rm SL}_2(R)$. $($Note that $Ch \in {\rm E}_2(R)$ if $\varepsilon \in
{\E}_2(R)$.$)$
\end{lemma}
\begin{proof}
The first identity is easily verified and the
second one follows immediately from it.
\end{proof}

\begin{lemma} \label{6} Let $R$ be a commutative ring.
One  has the identities, for $a, b, \lambda, \mu, x, y \in R$, 
$$
\E \begin{pmatrix} \lambda a & \lambda a\\ \mu a & \mu a\end{pmatrix} =
\E \begin{pmatrix} x & x\\ x & x\end{pmatrix}
\E \begin{pmatrix} y & y\\ -y & -y\end{pmatrix}_{I_2 + C}, \eqno{(2)}
$$
where $x = \frac{\lambda a + \mu a}{2}$, $y = \frac{\lambda a - \mu
a}{2}$, and where
\begin{eqnarray*}
I_2 + C &=& \begin{pmatrix} 1 + 2xy & 2xy\\ -2xy & 1 - 2xy\end{pmatrix} \in {\E}_2 (R).
\end{eqnarray*}
$$
\E\begin{pmatrix} \lambda b & - \lambda b\\ \mu b & - \mu b\end{pmatrix} =
\E \begin{pmatrix} x & -x \\ x & -x\end{pmatrix}
\E \begin{pmatrix} -y & y\\ y & -y\end{pmatrix}_{I_2 + B}, \eqno{(3)}
$$
where $x = \frac{\lambda b + \mu b}{2}$, $y = \frac{\mu b - \lambda
b}{2}$, and where
\begin{eqnarray*}
I_2 + B &=& \begin{pmatrix} 1 + 2xy & -2xy\\ 2xy & 1 - 2xy\end{pmatrix} \in {\E}_2 (R).
\end{eqnarray*}
\end{lemma}
\begin{proof} This is a direct verification.
\end{proof}
\begin{lemma} \label{7}
One has an expression of the form, for $\lambda, \mu, x, y \in R$, for $\delta \in {\rm SL}_2(R)$,

\begin{eqnarray*}
(\delta \perp I_{2n-2})\E\begin{pmatrix} \lambda x & \lambda y\\
\mu x & \mu y\end{pmatrix}(\delta \perp I_{2n-2})^{-1} &=& \E
\begin{pmatrix}\lambda' x & \lambda' y\\ \mu'x & \mu'
y\end{pmatrix},
\end{eqnarray*}
for some $\lambda'$, $\mu' \in R$.
\end{lemma}
\begin{proof} Take $\lambda' = \lambda a + \mu b$, $\mu' = c \lambda + \mu d$,
where $e_1 \delta = (a, b)$, $e_2\delta = (c, d)$.
\end{proof}

We now come to the Initial Structure Theorem:

\begin{theorem}\label{thm}{\rm \textbf{(Initial Structure Theorem)}}\\[2mm]
Let  $R$ be a commutative ring.  Then
{\rm ESp}$_{2n}(R)$ is
generated by matrices of type $\E_2 (R)$, $\E(A)$, $\E(B)$, $\E(C)$, $\E(D)$,
$I_2 + C$, $I_2 + B$, the last two being ``suitably placed''
elementary symplectic matrices.

Furthermore, one has a decomposition of a $\varepsilon_{\psi} \in
{\rm ESp}_{2n}(R)$ as a product of the type,

\begin{eqnarray}
\varepsilon_{\psi} & = & \small{^{\delta} \begin{pmatrix}A + C\\ D + C\end{pmatrix}
\begin{pmatrix}B +
D\\ B + A\end{pmatrix} \cdots \begin{pmatrix} B + D \\ B + A\end{pmatrix}}
\nonumber\\ & & \\
{} & = & \small{^{\delta} \begin{pmatrix}A + C \\ 0\end{pmatrix}
\begin{pmatrix} 0 \\ D + C \end{pmatrix}
\begin{pmatrix} B + D \\ 0\end{pmatrix} \begin{pmatrix} 0 \\ B + A \end{pmatrix}
  \cdots \begin{pmatrix}A + C \\
0\end{pmatrix} \begin{pmatrix}0 \\ D + C \end{pmatrix}} \nonumber \\
& & \\
{} & = & \tiny{^{\delta}\E(A)\E(C)_{I_2 + C}\E(B)\E(D)_{I_2 + B} \cdots
\E(A)\E(C)_{I_2 + C} \E(B)\E(D)_{I_2 + B}}  \nonumber \\
{} & & \\
{} & = & (\delta_1 \perp  \cdots \perp \delta_n)\E(A)\E(C)\E(B)\E(D) \cdots
 \E(A)\E(C) \E(B)\E(D) \nonumber\\
{} & & \\
{} & = & ^\delta\E(A)\E(C)\E(B)\E(D) \cdots \E(A)\E(C) \E(B)\E(D)\nonumber\\
& &
\end{eqnarray}
for some $\delta, \delta_1, \ldots, \delta_n \in {\rm E}_2(R)$
\end{theorem}
Proof:  Via standard commutator relations one can show that ESp$_{2n}(R)$  
is generated by means of the elementary symplectic
generators $E_{2 1}(x)$, $E_{12}(x)$, $S_{1 i}(x)$, $S_{2 j}(x)$, for $x \in
R$, $3 \leq i, j \leq 2n$.

One has a similar identity  to that of Lemma \ref{4}
on conjugating an expression of type $\prod_{j = 3}^{2n}S_{2j}(y_j)$.
Therefore, via Lemma 2.9, if $\varepsilon_{\psi} \in {\rm ESp}_{2n}(R)$, then
$\varepsilon_{\psi}$ can be written as a product of the form,
$$
 \varepsilon_{\psi} =
^{\delta} \prod \E\begin{pmatrix} \lambda x & \lambda y\\ \mu x & \mu y
\end{pmatrix}, \eqno{(4)}
$$
for some $\delta \in$~ E$_2(R)$, $x$'s, $y$'s, $\lambda$'s, $\mu$'s in $R$.
This is got by ``moving'' all the $E_{21}(*)$'s and $E_{12}(*)$'s
occuring to the ``left''. (In fact, one can also assert that each
$(\lambda, \mu) \sim_{E} (1, 0)$.)

In the notation of Lemma \ref{5}
$$
\E\begin{pmatrix} \lambda x & \lambda y \\ \mu x & \mu y \end{pmatrix} =
^{Ch} \E \begin{pmatrix} \lambda a & \lambda a \\ \mu a & \mu a\end{pmatrix}
\E \begin{pmatrix} \lambda b & -\lambda b \\ \mu b & - \mu b \end{pmatrix}, \eqno{(5)}
$$
where $Ch$ is the matrix given in Lemma \ref{5}. Note that $Ch$ is an
elementary matrix in E$_2(R)$ as $(\lambda, \mu) \sim_{E} (1, 0)$.

Now pooling together the three identitites (1-3) we get an expression, for
$\lambda, \mu, x, y \in R$,
$$
\E\begin{pmatrix} \lambda x & \lambda y \\ \mu x & \mu y\end{pmatrix} =
^{Ch}\E(A)\E(C)_{I_2 + C}\E(B)\E(D)_{I_2 + B}. \eqno{(6)}
$$
Now Identity (C) follows from Equation (5) via Lemma 2.8, Lemma 2.9. 

\medskip

Proof of Identity (D): We begin with the situation in Identity
(C). Using the equations in Lemma \ref{8} (see later) we can move the matrices
of the type $\{I_{2s}\perp (I_2 + B)\perp I_{2t}\}$, $2s + 2 + 2t
= 2n$, $\{I_{2s} \perp (I_2 + C)\perp I_{2t}\}$, for some $s\geq
1$, $t \geq 0$, to the left. Lemma \ref{8} shows that this is
possible without changing the form too much. Sometimes there are
additional terms input, but this term does not hinder us from
taking the matrix to the left. Moreover, the process is finite and
terminates with an element of the type $(\delta_1 \perp  \ldots
\perp \delta_n)$ as required.  we may assume that $\delta_2, 
\ldots, \delta_n$ are products of elements of the form
$I_2 + B$ and $I_2 + C$.

\medskip

We will deduce Identity (E) from Identity (D) by using the
commutator relations of type $[\E(X), \E(Y)]$.

\begin{lemma} \label{commutator}
We record all the possible commutator relations of type\\
 $[\E(X_i)(x), \E(Y_j)(y)]$, where $ X,Y\,\in \{ A,B,C,D\} $
  and $i,j \in \{2,3,\dots ,n\} $.
\begin{equation*}
 [ \E(X_i)(x),\; \E(X_j)(y)] = \; I_{2n}  \text{, for all  $ i,\, j $ where
$ X \in \{A,\:B,\:C,\:D \} $. }\quad\quad\quad\quad\quad
\end{equation*}
\begin{eqnarray*}
[\E(A_i)(x),\; \E(B_j)(y)]=
 \begin{cases}
              I_{2n}, &\text{ if  $ i \neq j $;}\quad\quad\quad\quad\quad\quad\\
             \{ I_2 + B(4xy) \} \perp I_{2n-2}, &\text{ if  $ i= j $. }
\end{cases}
\end{eqnarray*}
\begin{eqnarray*}
[\E(A_i)(x),\; \E(C_j)(y)]=
\begin{cases}
    I_{2(i-1)}  \perp  \{\E(C_{j-i+1})(-2xy)\} \perp I_{2(n-j)},  &\text{ if  $ i < j $;}\\
    I_{2(i-1)}  \perp \{ I_2 +C(-4xy)\} \perp  I_{2(n-i)}, &\text{ if  $ i= j $; }\\
    I_{2(j-1)}  \perp  \{\E(C_{i-j+1})(-2xy)\} \perp I_{2(n-i)},  &\text{ if  $ i > j $.}
\end{cases}
\end{eqnarray*}
\begin{eqnarray*}
[ \E(A_i)(x),\; \E(D_j)(y)]=
\begin{cases}
      I_{2(i-1)}  \perp  \{\E(D_{j-i+1})(-2xy)\} \perp I_{2(n-j)}, &\text{ if  $ i < j $; }\\
      I_{2(j-1)}  \perp  \{\E(D_{i-j+1})(-2xy)\} \perp I_{2(n-i)}, &\text{ if  $ i > j $. }
\end{cases}
\end{eqnarray*}
\begin{eqnarray*}
[\E(B_i)(x),\; \E(A_j)(y)]=
\begin{cases}
         I_{2n}, & \text{ if  $ i \neq j $;}\quad\quad\quad\quad\quad\quad\\
        \{ I_2 + B(-4xy) \} \perp I_{2n-2},&\text{ if  $ i= j $. }
\end{cases}
\end{eqnarray*}

\begin{eqnarray*}
[ \E(B_i)(x),\; \E(C_j)(y)]=
\begin{cases}
       I_{2(i-1)}  \perp  \{\E(A_{j-i+1})(2xy)\} \perp I_{2(n-j)}, &\text{ if  $ i < j $; }\\
       I_{2(j-1)}  \perp  \{\E(A_{i-j+1})(2xy)\} \perp I_{2(n-i)}, &\text{ if  $ i > j $. }
\end{cases}
\end{eqnarray*}

\begin{eqnarray*}
[ \E(B_i)(x),\; \E(D_j)(y)]=
\begin{cases}
    I_{2(i-1)}  \perp  \{\E(B_{j-i+1})(-2xy)\} \perp I_{2(n-j))}, &\text{ if  $ i < j $;}\\
    I_{2(i-1)}  \perp \{ I_2 +B(-4xy)\} \perp I_{2(n-i)}, &\text{ if $i= j $; }\\
    I_{2(j-1)}  \perp  \{\E(B_{i-j+1})(-2xy)\} \perp I_{2(n-i))}, &\text{ if  $ i > j $.}
\end{cases}
\end{eqnarray*}
\begin{eqnarray*}
[ \E(C_i)(x),\; \E(A_j)(y)]=
\begin{cases}
     I_{2(i-1)}  \perp  \{ \E(C_{j-i+1})(2xy)\} \perp I_{2(n-j)}, & \text{ if  $ i < j $;}\\
     I_{2(i-1)}  \perp \{ I_2 + C(4xy)\} \perp  I_{2(n-i)}, &\text{ if  $ i = j $;} \\
    I_{2(j-1)}  \perp  \{ \E(C_{i-j+1})(2xy)\} \perp I_{2(n-i)}, &\text{ if  $ i > j $.}
\end{cases}
\end{eqnarray*}
\begin{eqnarray*}
[ \E(C_i)(x),\; \E(B_j)(y)]=
\begin{cases}
      I_{2(i-1)}  \perp  \{\E(D_{j-i+1})(-2xy)\} \perp I_{2(n-j)}, &\text{ if  $ i < j $; }\\
      I_{2(j-1)}  \perp  \{\E(D_{i-j+1})(-2xy)\} \perp I_{2(n-i)}, &\text{ if  $ i > j $. }
\end{cases}
\end{eqnarray*}
\begin{eqnarray*}
[ \E(C_i)(x),\; \E(D_j)(y)]=
\begin{cases}
      I_{2n}, & \text{ if  $ i \neq j $; }\quad\quad\quad\quad\quad\quad\\
     \{ I_2 + C(4xy) \} \perp I_{2n-2}, &\text{ if  $ i= j $. }
\end{cases}
\end{eqnarray*}
\begin{eqnarray*}
[ \E(D_i)(x),\; \E(A_j)(y)]=
\begin{cases}
       I_{2(i-1)}  \perp  \{\E(A_{j-i+1})(2xy)\} \perp I_{2(n-j)}, &\text{ if  $ i < j $; }\\
      I_{2(j-1)}  \perp  \{\E(A_{i-j+1})(2xy)\} \perp I_{2(n-i)}, &\text{ if  $ i > j $. }
\end{cases}
\end{eqnarray*}
\begin{eqnarray*}
[ \E(D_i)(x),\; \E(B_j)(y)]=
\begin{cases}
     I_{2(i-1)}  \perp  \{\E(B_{j-i+1})(2xy)\} \perp I_{2(n-j)},  &\text{ if  $ i < j $;}\\
     I_{2(i-1)}  \perp \{ I_2 + B(4xy)\} \perp  I_{2(n-i)}, &\text{ if  $ i= j $; }\\
     I_{2(j-1)}  \perp  \{\E(B_{i-j+1})(2xy)\} \perp I_{2(n-i)},  &\text{ if  $ i > j $.}
\end{cases}
\end{eqnarray*}

\begin{eqnarray*}
[ \E(D_i)(x),\; \E(C_j)(y)]=
\begin{cases}
     I_{2n}, & \text{ if  $ i \neq j $; }\quad\quad\quad\quad\\
     \{ I_2 + C(-4xy) \} \perp I_{2n-2}\}, &\text{ if  $ i = j $.}
\end{cases}
\end{eqnarray*}
\begin{align*}
[ \E(A_2)(x),\; \E(D_2)(y) ]= \quad\quad\quad\quad   &   \\
 & \hspace{-3.6cm} {\scriptsize\left(\begin{array}{cc}
{I_2+ A(8x^2y^2)+A(2xy)+D(2xy)}
&D(4xy^2)+A(-4x^2y)\\
A(4xy^2)+ D(-4x^2y) & I_2+D(-8x^2y^2) + D(-2xy) +A(-2xy)
\end{array}\right)}   \\
 [ \E(B_2)(x),\; \E(C_2)(y)]  =\quad\quad\quad\quad  &  \\
 & \hspace{-3.5cm}{\scriptsize \left( \begin{array}{cc}
I_2+A(2xy)+D(2xy)+A(8x^2y^2) & B(-4x^2y)+C(4xy^2)\\
 B(-4x^2y)+C(4xy^2) & I_2 +A(2xy)+D(2xy)+A(8x^2y^2)
\end{array}\quad\quad \right )}\\
 [\E(C_2)(x),\; \E(B_2)(y)] =\quad\quad\quad\quad & \\
 & \hspace{-4.2cm}{\scriptsize \left( \begin{array}{cc}
I_2+A(-2xy)+D(-2xy)+D(-8x^2y^2)& B(4xy^2)+C(-4x^2y)\\
B(4xy^2)+C(-4xy^2) & I_2+A(-2xy)+D(-2xy)+D(-8x^2y^2)
\end{array}\right)}\\
[\E(D_2)(x),\; \E(A_2)(y) ] =\quad\quad\quad\quad  & \\
 & \hspace{-3.5cm}{\scriptsize \left(  \begin{array}{cc}
I_2+A(-2xy)+D(-2xy)+D(-8x^2y^2) & A(4xy^2)+D(-4x^2y)\\
D(4xy^2)+A(-4x^2y)& I_2 + A(2xy) + D(2xy)+ A(8x^2y^2)
\end{array}\right )}\\
\end{align*}
\hfill $\square$
\end{lemma}

\begin{cor}\label{smaller}
The elements
$\{I_{2r} \perp (I_2 + Y) \perp I_{2s}\}$, $r\geq 0$, $s \geq 0$,
$2r + 2 + 2s = 2n$, $Y = B$ or $C$, are in the
subgroup of {\rm ESp}$_{2n}(R)$ generated by elements of the form $\E(A)$,
$\E(B)$,  $\E(C)$,  $\E(D)$.
\end{cor}
\textbf{Proof}: The relations in Lemma \ref{commutator} show that
the ``smaller'' sized matrices are in the required subgroup. \hfill
$\square$

\medskip

A matrix of type $(\delta_1 \perp \ldots \perp \delta_n)$, as in
Identity (D)  can be written as a product of type
$(\delta_1 \perp I_{2n -2}) H$.

Hence, Identity (E) follows from Identity (D). \hfill $\square$

\medskip

L.N. Vaser{\v{s}}te{\u\i}n showed in \cite{SV} that the first row of an elementry
matrix of even size is the first row of an elementary symplectic matrix, i.e.
$e_{2n}{\E}_{2n}(R) = e_{2n}{\rm ESp}_{2n}(R)$. In view of this,
as a consequence of the initial Structure theorem we get:

\begin{cor} \label{haction}
Let $R$ be a commutative ring with $1$.
Then for elementary matrix $\varepsilon $, we have 
$e_{2n}\varepsilon = e_{2n} \alpha$, for some elementary symplectic
$\alpha$ in the subgroup $H$ of {\rm ESp}$_{2n}(R)$. 
\end{cor}
\begin{proof} In view of Corollary \ref{smaller} the initial Structure
Theorem asserts that if $\varepsilon_\psi \in$ ESp$_{2n}(R)$ then
$\varepsilon_\psi = ^\delta h$, for some $\delta \in {\E}_2(R)$, $h
\in H$. By Vaser{\v{s}}te{\u\i}n's lemma (\cite{SV}, Lemma 5.6)
$e_{2n}\varepsilon = e_{2n}\varepsilon_\psi$, for some
$\varepsilon_\psi \in {\rm ESp}_{2n}(R)$. Hence $e_{2n}\varepsilon
= e_{2n}(^\delta h) = e_{2n}\alpha$, for some $\alpha \in H$,
as required. \end{proof}

\begin{lemma}\label{sl2reln}
 Let $ \delta=\left(\begin{array}{lr} p & q \\
                                   r & s
             \end{array}\right) $
 with $ \det\delta=1 $. Let  $i \in \{2,3,\dots ,n\} $.
 Then we have the following identities:
\[
{}^{\delta}E(A_i)(x)^{\delta^{-1}}
        = \E^i\begin{pmatrix} px & px \\ rx & rx\end{pmatrix}
          \E^i\begin{pmatrix} qx & qx\\ sx & sx\end{pmatrix}\{I_{2i-2} \perp \{I_2 + C(-x^2)\} \perp I_{2n-2i}\}.
\]
\[
{}^{\delta}\E(B_i)(x)^{\delta^{-1}}
        = \E^i\begin{pmatrix} px & -px \\ rx & -rx\end{pmatrix}
          \E^i\begin{pmatrix} qx & -qx\\ sx & -sx\end{pmatrix}\{I_{2i-2} \perp \{I_2 + B(x^2)\} \perp I_{2n-2i}\}.
\]
\[
{}^{\delta}\E(C_i)(x)^{\delta^{-1}}
        = \E^i\begin{pmatrix} px & px \\ rx & rx\end{pmatrix}
          \E^i\begin{pmatrix} -qx & -qx\\ -sx & -sx\end{pmatrix}\{I_{2i-2} \perp \{I_2 + C(x^2)\} \perp I_{2n-2i}\}.
\]
\[
{}^{\delta}\E(D_i)(x)^{\delta^{-1}}
        = \E^i\begin{pmatrix} -px & px \\ -rx & rx\end{pmatrix}
          \E^i\begin{pmatrix} qx & -qx\\ sx & -sx\end{pmatrix}\{I_{2i-2} \perp \{I_2 + B(-x^2)\} \perp I_{2n-2i}\}.
\]
\hfill $\square$
\end{lemma}

\begin{cor}\label{sl2normH}  Let $\delta \in {\rm SL}_2(R)$, $\varepsilon \in
H(R)$, the subgroup of {\rm ESp}$_{2n}(R)$. 
Then $^\delta \varepsilon ^{\delta^{-1}} \in H$.
\end{cor}
\begin{proof} This is clear from Lemma \ref{sl2reln}, Equations (2)-(3),
and Corollary \ref{smaller}. \end{proof}

\begin{cor} Let $\delta \in {\rm SL}_2(R)$, $\varepsilon_\psi \in {\rm ESp}_{2n}(R)$.
Then $^\delta \varepsilon_\psi^{\delta^{-1}} = ^{\varepsilon}\alpha$, for some
$\alpha \in H$, $\varepsilon \in {\E}_2(R)$.
\end{cor}
\begin{proof} This follows from Lemma \ref{4}, Equations (1)-(3), Corollary
\ref{smaller}, and Lemma \ref{sl2reln}.
\end{proof}

\section{The final Structure Theorem}
Our final Structure Theorem is to assert that ESp$_{2n}(R)$, $n
\geq 2$, is generated by elements of the
type $\E(A)$, $\E(B)$, $\E(C)$, $\E(D)$.

Before we come to the final Structure Theorem we make a simple observation:

\begin{cor} \label{e2subH}
The subgroup ${\E}_2(R) \perp I_{2n - 2}$ is contained in $H$.
\end{cor}

\begin{proof} Let $\gamma = I_2 + B(c)$.
Then $\E_{12}(-1)\gamma \E_{12}(1)  = E_{21}(c)$. By Corollary \ref{sl2normH}
$\;\E_{12}(-1)\gamma \E_{12}(1) \in H$, hence  $\E_{21}(c) \in H$. Similarly, one
can show that  $\E_{12}(c) \in H$.
\end{proof}

We now come to the {\textbf{main Structure Theorem for the elementary symplectic
group}} of size atleast four:

\begin{theorem} \label{Main Theorem}
 For $n \geq 2$, {\rm ESp}$_{2n}(R)$ coincides with the subgroup
$H$.
\end{theorem}
\begin{proof} By the initial Structure Theorem Identity (E) it follows that
if $\varepsilon_\psi \in {\rm ESp}_{2n}(R)$ then $\varepsilon_\psi = ^\delta
h$, for some $\delta \in {\E}_2(R)$, $h \in H$.

Therefore, it suffices to show that $(\delta \perp I_{2n - 2}) \in
H$, for $\delta \in {\E}_2(R)$. This was shown above in
Corollary \ref{e2subH}.
\end{proof}

\medskip

\section{\large Local Global Principle for the $A, B, C, D$ subgroup}

In this section we give an alternate proof of, 
${\rm ESp}_{2n}(R)$ is a normal subgroup of ${\rm Sp}_{2n}(R)$,
from that of V.I. Kope{\u\i}ko in \cite{Kop} and G. Taddei in \cite{Taddei}. 
This proof will throw more light on the commutator relations between
 the special generators of type $A$, $B$, $C$, $D$ described above; 
which we feel is useful to record here. 

A sketch of the proof: By Theorem \ref{Main Theorem} we have 
$H={\rm ESp}_{2n}(R)$ for $n\geq 2.$
We prove that $H$ is a normal subgroup of ${\rm Sp}_{2n}(R)$, $n \geq 2$.
Our idea to prove this is to establish that $H(R[X])$ satisfies
the Local Global principle enunciated by D. Quillen in \cite{Q} to
settle the Serre's problem on projective modules over a polynomial
ring. Our treatment to establish this Local Global principle is
influenced by A. Suslin's treatment in \cite{18}, which in turn
was inspired by D. Quillen's approach in \cite{Q}. The treatment of 
V.I. Kope{\u\i}ko in \cite{Kop} is also inspired by \cite{18}; however 
our treatment is via commutator laws (and not the special forms as 
in \cite{18}, \cite{Kop})) and is similar to the treatment in \cite{acr} 
for the relative groups via commutator relations.

\begin{lemma}\label{8}
We record the commutator relations of type  $[E(X_i),
\{I_{2j-2} \perp \{I_2+Y\}  \perp I_{2n-2j} \}]$, with $X = A, B,
C, D $ and  $ Y = B, C $ where  $i \in
\{2,3,\dots ,n\} $ and $ j \in \{1,2,...,n\}.$

\[
\begin{split}
 [ \E(B_i)(x), \{I_{2j-2} \perp \{I_2+B(y)\} \perp I_{2n-2j} \} ]& = I_{2n}, \text{ for all  $ i,j $.}\quad\quad\quad
 \quad\quad\quad\quad\quad \\
 [ \E(C_i)(x), \{I_{2j-2} \perp \{I_2+C(y)\} \perp I_{2n-2j} \} ]& = I_{2n}, \text{ for all  $ i,j $.}
 \end{split}
\]
\begin{eqnarray*}
 \lefteqn{[ \E(A_i)(x), \{I_{2j-2} \perp \{I_2+B(y)\} \perp I_{2n-2j} \}]
 =}\\
&& \quad\quad\quad\quad\quad\quad\quad\quad\quad \begin{cases} \{ \{I_2 +B(4x^2y)\} \perp  I_{2n-2i} \} \E(B_i)(2xy), & \text{ if  $ i = j $; }\\
              I_{2n},  & \text{ if  $ i \neq j $.}
\end{cases}
\end{eqnarray*}

\begin{eqnarray*}
 \lefteqn{[ \E(A_i)(x), \{I_{2j-2} \perp \{I_2+C(y)\} \perp I_{2n-2j}
 \}]=}\\
&& \hspace{-0.5cm}\begin{cases} \{I_{(2i-2)} \perp \{ I_2 +C(4x^2y)\} \perp  I_{2n-2i} \} \E(C_i)(-2xy), & \text{ if  $ j = 1 $; }\\
              I_{2n},  & \text{ for any i,j with   $ j \neq 1 $.}\\
\end{cases}
\end{eqnarray*}

\begin{eqnarray*}
 \lefteqn{[ \E(B_i)(x), \{I_{2j-2} \perp \{I_2+C(y)\} \perp I_{2n-2j} \}] =}\\
&& \quad \begin{cases}
            \{I_{(2i-2)} \perp \{ I_2 +B(-4x^2y)\} \perp  I_{2n-2i} \} \E(D_i)(-2xy), & \text{ if  $ j = 1 $; }\\
            \{\{I_2 +B(-4x^2y)\} \perp  I_{2n-2} \} \E(A_i)(2xy),  & \text{ if  $ i = j $; }\\
              I_{2n},  & \text{ if  $ i \neq j,j\neq1 $.}
\end{cases}
\end{eqnarray*}
\begin{eqnarray*}
 \lefteqn{[ \E(C_i)(x), \{I_{2j-2} \perp \{I_2+B(y)\} \perp I_{2n-2j} \}]=}\\
&& \quad\quad \hspace{-0.5cm}\begin{cases}
              \{I_{(2i-2)} \perp \{ I_2 +C(-4x^2y)\} \perp  I_{2n-2i} \} \E(A_i)(-2xy), & \text{ if  $ j = 1 $; }\\
             \{\{I_2 +C(-4x^2y)\} \perp  I_{2n-2} \} \E(D_i)(2xy),  & \text{ if  $ i = j $; }\\
              I_{2n}, & \text{ if  $ i \neq j,j\neq1 $.}
\end{cases}
\end{eqnarray*}

\begin{eqnarray*}
 \lefteqn{[ \E(D_i)(x), \{I_{2j-2} \perp \{I_2+B(y)\} \perp I_{2n-2j} \}] =}\\
&& \quad\quad\quad\quad \begin{cases}
             \{I_{(2i-2)} \perp \{I_2 +B(4x^2y)\} \perp  I_{2n-2i} \} \E(B_i)(-2xy), & \text{ if  $ j = 1 $; }\\
              I_{2n},  & \text{ if $ j \neq 1 $.}\\
\end{cases}
\end{eqnarray*}

\begin{eqnarray*}
\lefteqn{[ \E(D_i)(x), \{I_{2j-2} \perp \{I_2+C(y)\} \perp I_{2n-2j} \}]=}\\
&& \quad\quad\quad\quad \begin{cases}
              \{I_{(2i-2)} \perp \{I_2 +B(4x^2y)\} \perp  I_{2n-2i} \} \E(C_i)(-2xy), & \text{ if  $ j = 1 $; }\\
              I_{2n},  & \text{ if $  j \neq 1 $.}
\end{cases}
\end{eqnarray*}
\hfill $\square$
\end{lemma}

\begin{lemma}
We  record the commutator relations of type  $[\, \E(X_i)(x)\,
,\,\E(Y_i)(y)\, ]$, with $ (X,\:Y) \in \{ (A,D),\;(B,C)\} $,  where $i
\in \{2,3,\dots ,n\} $.
\[
\begin{split}
[\: \E(&A_i)(x),\; \E(D_i)(2yz)\:]\\
=&[ \E(A_i)(x),\; \{ \{I_2+C(4y^2z)\} \perp I_{2n-2}\} ]\; \{ \{I_2+C(4y^2z)\} \perp I_{2n-2}\}\\
 &[[\E(A_i)(x),\:\E(C_i)(y)] E(C_i)(y),\;[ \E(A_i)(x),\: \{I_{2i-2} \perp \{I_2+B(z)\} \perp I_{2n-2i} \}]\\
 &\{I_{2i-2} \perp \{I_2+B(z)\} \perp I_{2n-2i} \}]\\
 &[ \{I_{2i-2} \perp \{I_2+B(z)\} \perp I_{2n-2i} \},\;\E(C_i)(y)]\{ \{I_2+C(-4y^2z)\} \perp I_{2n-2}\}.\\
 \end{split}
 \]
 \[
\begin{split}
[\: \E(&B_i)(x),\;\E(C_i)(2yz)\:]\\
=&[ \E(B_i)(x),\;\{ I_{(2i-2)}\perp\{I_2 +C(4y^2z)\}\perp  I_{2n-2i}\}]\\
 &\{ I_{(2i-2)}\perp\{I_2 +C(4y^2z)\}\perp  I_{2n-2i}\}[[\E(B_i)(x),\; \E(A_i)(y)] \E(A_i)(y),\\
 &[ \E(B_i)(x),\;\{I_2+C(-z)\} \perp I_{2n-2} \}]\{\{I_2+C(-z)\} \perp I_{2n-2} \}]\\
 &[\{I_2+C(-z)\} \perp I_{2n-2} \},\;\E(A_i)(y)]\{ I_{(2i-2)}\perp\{I_2 +C(-4y^2z)\}\perp  I_{2n-2i}\}.\\
\end{split}
 \]
\end{lemma}
\begin{proof} By Lemma \ref{8}, we have
\[
\begin{split}
&\E(D_i)(2yz)=\{ \{I_2+C(-4y^2z)\} \perp I_{2n-2}\} ^{-1}[\E(C_i)(y),\; \{I_{2i-2} \perp\\
 & \{I_2+B(z)\} \perp I_{2n-2i} \}]\text{ and }\\
\end{split}
\]
 \[
\begin{split}
&\E(C_i)(2yz)= \{ I_{(2i-2)}\perp\{I_2 +C(-4y^2z)\}\perp
I_{2n-2i}\}^{-1} [ \E(A_i)(y),\; \\
&\{\{I_2+C(-z)\} \perp I_{2n-2} \}].
\end{split}
\]
Now using the commutator formula
\[
  [\: x,\: y[z, w]\: ]= [x,y]\: y\: [\: [x,z]z,\: [x,w]w \:]\: [w,z]\: y^{-1},  
\]
results follows.
\end{proof}
\iffalse

\begin{lemma}
\[
\begin{split}
[ [ E(A_2(a)),E(B_2(b)) ],E(D_2(d))] & =  \{I_2 \perp \{I_2 +B(-16abd^2)\}\}E(B_2(-8abd)). \\
[ [ E(A_2(a)),E(C_2(c)) ],E(D_2(d))] & = \{ \{I_2+C(-16acd^2)\}  \perp  I_2 \}E(C_2(-8acd)).\\
[ [ E(B_2(b)),E(A_2(a)) ],E(C_2(c))] & = \{I_2 \perp \{I_2 +C(-16abc^2)\}\}E(A_2(-8abc)).\\
[ [ E(B_2(b)),E(D_2(d)) ],E(C_2(c))] & = \{\{ I_2 +C(-16bc^2d)\}  \perp  I_2 \}E(D_2(-8bcd)).\\
[ [ E(D_2(d)),E(B_2(b)) ],E(A_2(a))] & = \{\{I_2+ B(-16a^2bd)\}  \perp  I_2 \}E(B_2(-8abd)).\\
[ [ E(D_2(d)),E(C_2(c)) ],E(A_2(a))] & = \{I_2 \perp \{ I_2+C(16a^2cd \}\}E(C_2(-8acd)).\\
\end{split}
\]
\end{lemma}
\fi
\begin{prop}\label{product decomposition}
Let $ X ,Y , Z \in \{ A,B,C,D\} $. Let $k \in \mathbb{N}
 $ be fixed. Let $s$ be a non nilpotent element of $R$.
 Let $ m> k $. Let $ i,j,r_{t} \in \{2,3,.., n\} $ for every
 $ t \in \mathbb{N} $. Then there exists a product decomposition
in ${\rm ESp}_{2n}(R_{s})$ 
\[ 
\E(X_{i})(a/s^k)\E(Y_{j})(s^mx)\E(X_{i})(a/s^k)^{-1} =
\prod_{t=1}^{\lambda}\E(Z_{r_{t}})(s^{m_{t}}x_{t}),
\]
where $a$ and $x$ are elements of $R$ and the $x_t$ are suitable elements 
in $R$, also $\lambda \leq 45$, with $ m_t \rightarrow \infty $ as $m\rightarrow\infty$.
\end{prop}

\begin{proof}
In order to cover all the possibilities, we discuss the proof by considering
three cases as follows:\\
{\bf{Case $1$.}} When $ X \in \{A,\:B,\:C,\:D \} $, then by Lemma \ref{commutator} one has
\[
\E(X_{i})(a/s^k)\E(X_{i})(s^mx)\E(X_{i})(a/s^k)^{-1} = \E(X_i)(s^mx).
\]
For instance,  $\E(A_i)(a/s^k)\E(A_i)(s^mx)\E(A_i)(a/s^k)^{-1} = \E(A_i)(s^mx).$\\
{\bf{Case $2$.}} When $(X,Y)\in\{
  (A,B),\:(A,C),\:(B,A),\:(B,D),\:(C,A),\:(C,D),\:(D,B),\\ \:(D,C)\}$. Then by Lemma \ref{commutator} one has
\[
\E(X_{i})(a/s^k)E(Y_{j})(s^mx)\E(X_{i})(a/s^k)^{-1} =
\prod_{t=1}^{\lambda}\E(Z_{r_{t}})(s^{m_{t}}x_{t}), \text{  where  $\lambda \leq 5.$ } 
\]
For instance,
\[
\begin{split}
\E(A_i)(a/s^k)\E(B_j)(s^mx)\E(A_i)(a/s^k)^{-1}
           & =  [\E(A_i)(a/s^k),\E(B_j)(s^mx)]\E(B_j)(s^mx) \\
           & =  [\E(A_i)(as^p),\E(B_j)(s^qx)]\E(B_j)(s^mx) \\
           & =  \prod_{t=1}^{\lambda}\E(Z_{r_{t}})(s^{m_t}x_{t}) \text{ for $ m_t > 0 $,}\\
\end{split}
\]
 where $\lambda =1$, if $i\neq j$ and  $\lambda =5, $ if $i= j$
(The penultimate equation is via Lemma \ref{commutator}, and
holds  for any positive integers $p$,\: $q$, with $p+q=m-k$). Similarly 
\[
\begin{split}
\E(A_i)(a/s^k)\E(C_j)(s^mx)\E(A_i)(a/s^k)^{-1}
            &=  [ \E(A_i)(a/s^k),\E(C_j)(s^mx)]\E(C_j)(s^mx) \\
            &=  [ \E(A_i)(as^p),\E(C_j)(xs^q)]\E(C_j)(xs^m) \\
            &= \prod_{t=1}^{5}\E(Z_{r_{t}})(s^{m_{t}}x_{t}) \text{ for $ m_t> 0 $.}\\
\end{split}
\]
(The penultimate equation holds by Lemma \ref{commutator} for any positive integers
$p$,\; $q$, with $p+q=m-k$.)\\
{\bf{Case $3$.}} When $(X,Y)\in \{(A,D),\:(B,C),\:(D,A),\:(C,B))\}$.\\
Assume $i\neq j$. Then one has 
 \[
 \E(X_{i})(a/s^k)\E(Y_{j})(s^mx)\E(X_{i})(a/s^k)^{-1} =
              \prod_{t=1}^{5}\E(Z_{r_{t}})(s^{m_{t}}x_{t}).
\]
We work out the case $(X,Y)= (A,D)$ below. The other cases can be
worked out similarly.
 For instance,
\[
\begin{split}
\E(A_i)(a/s^k)\E(D_j)(s^mx)\E(A_i)(a/s^k)^{-1}
           & =  [ \E(A_i)(a/s^k),\E(D_j)(s^mx)]\E(D_j)(s^mx) \\
           & =  [ \E(A_i)(as^p) ,\E(D_j)(s^qx)]\E(D_j)(s^mx) \\
           & =  \prod_{t=1}^{5}\E(Z_{r_{t}})(s^{m_{t}}x_{t}) \text{ for $ m_t > 0 $.}\\
\end{split}
\]
Assume $i = j$. Then one has
\[
\E(X_{i})(a/s^k)\E(Y_{j})(s^mx)\E(X_{i})(a/s^k)^{-1} =
     \prod_{t=1}^{\lambda}\E(Z_{r_{t}})(s^{m_{t}}x_{t}), \text{  where $\lambda \leq 45$.}
\]
We work out the case $(X,Y)= (A,D)$ below. The other cases can be
worked out similarly. For instance,
\[
\begin{split}
[\: \E(&A_i)(x),\; \E(D_i)(2yz)\:]\:\E(D_i)(2yz)\\
=&[ \E(A_i)(x),\; \{ \{I_2+C(4y^2z)\} \perp I_{2n-2}\} ]\; \{ \{I_2+C(4y^2z)\} \perp I_{2n-2}\}\\
 &[\:[\E(A_i)(x),\:\E(C_i)(y)] \E(C_i)(y),\;[ \E(A_i)(x),\: \{I_{2i-2} \perp \{I_2+B(z)\} \perp I_{2n-2i} \}]\\
 &\{I_{2i-2} \perp \{I_2+B(z)\} \perp I_{2n-2i} \}][ \{I_{2i-2} \perp \{I_2+B(z)\} \perp I_{2n-2i} \},\;\E(C_i)(y)]\\
 &\{ \{I_2+C(-4y^2z)\} \perp I_{2n-2}\}\:\E(D_i)(2yz).\\
 \end{split}
 \]
We can write it as 
\[
\begin{split}
[ \E(&A_i)(x),\; \E(D_i)(2yz)]\:\E(D_i)(2yz)\\
=&[\: \E(A_i)(-4xz),\;\E(C_i)(xy^2)\:]\:\E(C_i)(-8xy^2z)\: [\E(C_i)(y^2),\;\E(D_i)(z)]\\
&[\:[\:\E(A_i)(x),\;\E(C_i)(y)\:] \E(C_i)(y),\;[\: \E(A_i)(x^2),\;\E(B_i)(z)\:]\:\E(B_i)(2xz)\;]\\
&\{I_{2i-2} \perp \{I_2+B(z)\} \perp I_{2n-2i} \}][ \E(C_i)(y),\; \E(D_i)(z)]\: \E(D_i)(2yz)\\
&[\:\E(C_i)(y^2),\;\E(D_i)(z)\:]\;\E(D_i)(2yz).\\
\end{split}
\]
Now choose $ x= a/s^k , y= s^m $ and $ z= 4us^{2m} $ where $a,u$ and $s$ are elements 
of $R$ with $s$ non nilpotent. Then $xz=4aus^{2m-k},\; xy=as^{m-k},\;
yz= 4us^{3m} $ and  $ xy^2z= 4aus^{4m-k}$. Note that by Lemma \ref{commutator} we may write 
$\{I_{2i-2} \perp \{I_2+B(z)\} \perp I_{2n-2i} \}  = [\E(D_i)(us^m),\;\E(B_i)(s^m)] $. Thus 
for some $x_t$ in $ R $, we conclude that 
 $$
[\: \E(A_i)(a/s^k),\; \E(D_i)(8us^{3m})\:]\:\E(D_i)(8us^{3m})
 = \prod_{t=1}^{\lambda}\E(Z_{r_{t}})(s^{m_{t}}x_t), 
$$ 
where $m_t$  is a positive
integer such that $m_t \rightarrow \infty $ as $ m\rightarrow\infty. $
\end{proof}

We record a well-known useful observation in the form of a lemma:

\begin{lemma}\label{Group identity}
 Let $G$ be a group and $a_i,\;b_i \in G,$  for $i=1,\dots,n.$ Then
\[
\prod_{i=1}^{n}a_ib_i=\prod_{i=1}^{n}r_ib_ir_i^{-1}\prod_{i=1}^{n}a_i,
\]
 where $r_i=\prod_{j=1}^{i}a_j.$\hfill $\square$
\end{lemma}

\begin{prop} \label{dilation} {\rm \textbf{(Dilation Principle)}}\\[2mm]
Let $R$ be a commutative ring. Let $s$ be a
non-nilpotent element of $R$. Let $\alpha(X) \in {\rm
Sp}_{2n}(R[X])$ with $ \alpha(0) = I_{2n} $. Let $Y,Z \in \{A,B,C,D\}$. 
If $\alpha_s(X) (=\alpha(X)_s) \in$ $\H(R_s[X])$, then for $m >> 0$, for all $b
\in (s)^m R$, one has $\alpha(bX) \in \H(R[X])$.
\end{prop}
\begin{proof} Let $ \alpha_s(X)  = \prod_{k=1}^{r}E(Y_{i_k})(b_k(X)) \in H(R_s[X]) $ 
where $ b_k(X) \in R_s[X] $ for all $k$, also $ i_{k} \in \{2,3,.., n\} $ for 
every $ k \in \mathbb{N} $. Let $ b_k(X) = b_k(0) + Xb\acute{}_k(X) $. 
Since $ \E(Y)(a+b)=\E(Y)(a)\E(Y)(b),$ we can write
\[
\alpha_s(X) = \prod_{k=1}^{r}
\E(Y_{i_k})(b_k(0))\E(Y_{i_k})(Xb\acute{}_k(X)).
\]
By Lemma \ref{Group identity}. one has  
\[
 \alpha_s(X) = \prod_{k=1}^{r} \gamma_k
\E(Y_{i_k})(Xb\acute{}_{k}(X)) \gamma_{k}^{-1} \prod_{k=1}^{r} \E(Y_{i_k})(b_{k}(0)), 
\]
where $ \gamma_{k} = \prod_{j=1}^{k} \E(Y_{i_j})(b_{k}(0)$, here $ i_{j} \in
\{2,3,.., n\} $ for every $ j \in \mathbb{N} $. As $ \gamma_r = I_{2n} $. Therefore one has
\[
\alpha_s(X)= \prod_{k=1}^{r} \gamma_k
\E(Y_{i_k})(Xb\acute{}_{k}(X)) \gamma_k^{-1}.
 \]
Hence we can write
\[
 \alpha_s(s^mX)= \prod_{k=1}^{r} \gamma_k
\E(Y_{i_k})(s^mXb\acute{}_{k}(s^mX)) \gamma_k^{-1}.
\]
Our next claim is that if $ \beta = \prod_{j=1}^{k} \E(Y_{i_j})(b_j), \; b_j \in R_s $, then we can
show that one has a product decomposition
\[
\beta \E(Z)(s^mx)\beta^{-1} =
 \prod_{t=1}^{\lambda_{k}}\E(Z_{r_{t}})(s^{m_t}x_t),\eqno{(7)}
\]
with $ m_t\rightarrow\infty $ as $m\rightarrow \infty $, for some $x_t \in R$. 
We do this by induction on $k$. We may write $\beta = \beta_{1}\beta_{2}\cdots \beta_{k}$, 
where $\beta_j = \E(Y_{i_j})(b_j)$. For $ k=1. $ By Proposition \ref{product decomposition} 
we have a product decomposition
\[
 \beta_1\E(Z)(s^mx)\beta_1^{-1} =
\prod_{t=1}^{\lambda_1}\E(Z_{r_{t}})(s^{m_t}x_t),
\]
with $ m_t\rightarrow\infty $ as $ m \rightarrow\infty $. Assume that the result is true for $k-1$, that is,
\[
 \beta_{1}\beta_{2}\cdots \beta_{k-1}\E(Z)(s^mx)(\beta_{1}\beta_{2}\cdots \beta_{k-1})^{-1}=
\prod_{t=1}^{\lambda_{k-1}}\E(Z_{r_{t}})(s^{m_t}x_t),
\]
with $ m_t\rightarrow\infty $ as $m\rightarrow \infty $, for some $x_t \in R$. 

By Proposition \ref{product decomposition}, we also have 
\[
 \beta_{k}\E(Z)(s^mx)\beta_{k}^{-1}=\prod_{t=1}^{\lambda}\E(Z_{r_{t}})(s^{m_t}x_t)
=\theta_1\theta_2\cdots \theta_{\lambda} \text{ (say).}
\]
For conveinence, call $ \beta_{}^{\prime} =  \beta_{1}\beta_{2}\cdots \beta_{k-1}$. Now it 
is enough to show that 
$ \beta_{}^{\prime} \theta_1\theta_2\cdots \theta_{\lambda}  {\beta_{}^{\prime}}^{-1} $ is of the 
form $(7)$. Since we may write
\[
 \beta_{}^{\prime} \theta_1\theta_2\cdots \theta_{\lambda}  {\beta_{}^{\prime}}^{-1}=
 \beta_{}^{\prime} \theta_1 {\beta_{}^{\prime}}^{-1} \beta_{}^{\prime} \theta_2  {\beta_{}^{\prime}}^{-1}  
 \cdots   \beta_{}^{\prime} \theta_{\lambda}  {\beta_{}^{\prime}}^{-1}.
\]
Thus the claim follows by applying induction hypothesis for $k-1$. 
Therefore we can write
\[
\alpha_s(s^mX) = \prod_{k=1}^{r}\prod_{t=1}^{\lambda_{k}}\E(Z_{r_{t}})(s^{m_t}x_t).
\]
For m large enough, the term $s^{m_t}x_t$ is contained in $R[X]$, as required. Hence
\[
\alpha(bX)=\prod_{k=1}^{r}\prod_{t=1}^{\lambda_{k}}\E(Z_{r_{t}})(s^{m_t}x_t) \in \H(R[X]). 
\]
\end{proof}

\begin{theorem} \label{LG} {\rm \textbf{(Local Global Principle)}} \\[2mm]
Let $\alpha(X) \in {\rm Sp}_{2n}(R[X]),\; $ with $\alpha(0)= I_{2n} $. If $
\alpha(X)_{\mathfrak{m}} \in \H(R_{\mathfrak{m}}[X])$, for all maximal
ideals ${\mathfrak{m}}$ of $R$, then $\alpha(X) \in \H(R[X])$.
\end{theorem}
\begin{proof}
Let ${\mathfrak{m}}$ be a maximal ideal of $R$. Choose an element
$ a_{\mathfrak{m}}$ from $ R\setminus {\mathfrak{m}} $ such that 
$\alpha(X)_{a_{\mathfrak{m}}} \in H(R_{a_{\mathfrak{m}}}[X])$. Let us define 
$\beta(X,Y)=\alpha(X+Y)_{a_{\mathfrak{m}}}\alpha(Y)_{a_{\mathfrak{m}}}^{-1}.$ 
Clearly $ \beta(X,Y) \in \H(R_{a_{\mathfrak{m}}}[X,Y])$, and $ \beta(0,Y) = I_{2n}$. 
Therefore by Proposition \ref{dilation}, we have $ \beta(b_{\mathfrak{m}}X,Y) \in \H(R[X,Y])$, 
where $b_{\mathfrak{m}} \in (a_{\mathfrak{m}}^N)$, for some $N>>0$.

The ideal generated by the $b_{\mathfrak{m}}$ is the whole  ring $R$.
 Therefore we have $c_1b_{{\mathfrak{m}}_1} + c_2b_{{\mathfrak{m}}_2}+
\ldots + c_kb_{{\mathfrak{m}}_k} = 1$, where $c_i \in R$, for
$ 1\leq i\leq k$. Note that $\beta(c_ib_{{{\mathfrak{m}}_i}}X,Y) \in \H(R[X,Y])$, for
$ 1\leq i\leq k$. We can write
\[
 \alpha(X)=\left( \prod_{i=1}^{k-1}\beta(c_ib_{{\mathfrak{m}}_i}X,T_i) \right) \beta(c_kb_{{\mathfrak{m}}_k}X,0),
\]
where $T_i= c_{i+1}b_{{\mathfrak{m}}_{i+1}}X + \ldots +
c_{k}b_{{\mathfrak{m}}_{k}}X$. Hence $\alpha(X) \in  \H(R[X])$.
 \end{proof}

\begin{cor} The subgroup $\H(R)$ (viz. {\rm ESp}$_{2n}(R)$ by 
Theorem \ref{Main Theorem}) is a normal subgroup of {\rm Sp}$_{2n}(R)$.
\end{cor}
\begin{proof} Let $\gamma \in $ Sp$_{2n}(R)$, $h \in \H(R)$. Choose a
homotopy $h(T) \in \H(R[T])$ of $h$. Consider $\gamma h(T) \gamma^{-1}$. 
Note that for a prime ideal ${\mathfrak{p}}$ of $R$, Sp$_{2n}(R_\mathfrak{p})
=$ ESp$_{2n}(R_\mathfrak{p})$. By Theorem 2.10 (E), $\gamma_{\mathfrak{p}} =
^{\delta(\mathfrak{p})} \alpha$, for some $\alpha \in
\H(R_{\mathfrak{p}})$. 

Thus,
 $\gamma_{\mathfrak{p}} h(T)_{\mathfrak{p}}\gamma_{\mathfrak{p}}^{-1} =
^\delta \alpha h(T)_{\mathfrak{p}}(^\delta \alpha)^{-1}$, for some $\alpha \in
\H(R_{\mathfrak{p}})$, $\delta \in {\E}_2(R_{\mathfrak{p}})$.
By Lemma \ref{sl2normH},
$\gamma_{\mathfrak{p}} h(T)_{\mathfrak{p}}\gamma_{\mathfrak{p}}^{-1}
\in \H(R_{\mathfrak{p}}[T])$, for all primes
${\mathfrak{p}}$ of $R$. By the Local Global Principle proved in Theorem \ref{LG} 
\begin{eqnarray*}
\gamma h(T)\gamma^{-1} &\in& \gamma h(0)\gamma^{-1}\H(R[T]) 
=\H(R[T]), 
\end{eqnarray*}
as $h(0)= I_{2n}$. Hence $\gamma h(T)\gamma^{-1} \in 
\H(R[T])$. Hence $\gamma
h(1)\gamma^{-1} \in \H(R)$, as required.
\end{proof}

%\newpage
%\chapter{References}
%\addcontentsline{toc}{chapter}{Bibliography}

\end{document}